\documentclass[12pt]{article}
\usepackage{t1enc,amssymb,amsbsy,amsthm,amsmath,amsfonts}
\usepackage{graphicx}
\usepackage{color}

\newcommand{\e}{\rm e}

\newtheorem{theorem}{Theorem}[section]
\newtheorem{lemma}[theorem]{Lemma}
\newtheorem{proposition}[theorem]{Proposition}
\newtheorem{corollary}[theorem]{Corollary}
\newtheorem{example}[theorem]{Example}

\newtheorem{remark}[theorem]{Remark}

\numberwithin{equation}{section}

\def\th{\hbox {\rm tanh\,}}

\def\ln{\hbox {\rm ln\,}}

\def\e{\hbox {\rm e}}

\def\P{{\bf P}}
\def\R{{\bf R}}
\def\R{{\bf R}}

\def\E{{\bf E}}
\def\ep{{\epsilon}}

\def\cF{{\cal F}}

\def\cU{{\cal U}}

\def\cA{{\cal A}}

\def\l{\left}
\def\r{\right}

\def\al{\alpha}
\def\ep{\varepsilon}

\def\de{\delta}
\def\th{\theta}

\def\rp{\right)}
\def\lp{\left(}

\begin{document}

\author{Paavo Salminen\\{\small AAbo Akademi University}
\\{\small Faculty of Science and Engineering}
\\{\small FIN-20500 AAbo, Finland} \\{\small email: phsalmin@abo.fi}
\and
Pierre Vallois
\\{\small Universit\'e de Lorraine}
\\{\small CNRS, INRIA, IECL}
\\{\small F-54000 Nancy, France}
\\{\small email: pierre.vallois@univ-lorraine.fr}
}

\title{Drawdowns of diffusions}


\maketitle

\begin{abstract}
In this paper we give excursion theoretical  proofs of Lehoczky's formula (in an extended form allowing a lower bound for the underlying diffusion)  for the joint distribution of the first drawdown time and the maximum before this time, and of Malyutin's formula for the joint distribution of the first hitting time and the maximum drawdown before this time. It is remarkable -- but there is a clean explanation -- that the excursion theoretical approach which we  developed first for Lehoczky's formula provides also a proof for Malyutin's formula.  Moreover, we discuss some generalizations and analyze the pure jump process  describing  the maximum before the first drawdown time when the  size of the drawdown is varying.  
\end{abstract}

\medskip
\noindent
{\bf{Keywords:}} {Brownian motion, Poisson point process, additive functional, local time, pure jump process.}

\medskip
\noindent
{\bf{MSC Classification:}} {60J60, 60J65, 60G40.}

\newpage

\section{Introduction}
\label{sec0}

To calculate probability distributions of  functionals of stochastic processes is an essential/central/important issue/topic per se but also, of course,  many applications of stochastic processes and stochastic models call for explicit formulas for such distributions. This theme is developed also in this paper. We focus here on one-dimensional diffusions and functionals  associated with the running maximum and the position of the underlying diffusion. In particular, we analyze various aspects and distributions  of drawdowns from the running maximum.  To be more precise, let  $X=(X_t)_{t\geq 0}$ denote a one-dimensional diffusion defined on an interval. Then  the running maximum process and the drawdown process are defined via
$$
M_t:=\max_{0\leq s\leq t} X_s\ \ {\rm{and}}\ \ DD_t:=M_t-X_t,
$$
respectively. We introduce also  the stopping time when the process first drops below the maximum by $\delta>0$ units as
\begin{equation}\label{dtheta}
   \theta_\delta:=\inf\{t\geq 0\, :\, M_t-X_t>\delta\};
\end{equation}
also called as the first drawdown time of size $\de$. Another object of our interest is the maximum drawdown up to time $t>0$  given by
\begin{equation}\label{MXD}
D^-_t:=\max_{0\leq s\leq t}(M_s-X_s).
\end{equation}

Lehoczky derived in \cite{lehoczky77} an expression for the joint Laplace transform of the variables of   $  \theta_\delta$ and $M_{\theta_\delta}$ when $X$ is generated via an It\^o SDE. The distribution for Brownian motion with drift was characterized earlier by Taylor  in \cite{taylor75}. Taylor's proof is  based on particular properties of Brownian motion; especially, on the spatial homogeneity. Lehoczky's approach uses a passage to the limit in a discrete setting. Fitzsimmons in \cite{fitzsimmons22} addresses Lehoczky's formula  via excursion theory. The first main theme of our paper is to present a proof of Lehoczky's formula via excursion theory. Our proof uses slightly different ``angle of attack'' than Fitzsimmons' approach in  \cite{fitzsimmons22} . Some misprints in \cite{fitzsimmons22} are also corrected.

 Another main issue in the paper is to study the maximum drawdown up to a hitting time, that is, the random variable $D^-_{H_\eta}$ where   $H_\eta$ stands for the first hitting time of the state $\eta$, i.e.,
\begin{equation}\label{Hit}
   H_\eta:=\inf\{t\geq 0\, :\,X_t=\eta\}.
\end{equation}
We determine  the joint distribution of $H_\eta$ and $D^-_{H_\eta}$.
 The distribution of  $D^-_{H_\eta}$ for Brownian motion with drift was  calculated by Brockwell in \cite{brockwell74} via a passage to the limit in a discrete setting. Malyutin in \cite{malyutin87}   treated more general processes satisfying  regularity conditions on some key distributions of the underlying process.  He then deduced  the joint distribution of $H_\eta$ and $D^-_{H_\eta}$ when passing to the limit in a discrete setting. We demonstrate that the  joint distribution can be otained  via the excursion theoretical approach developed here to prove Lehoczky's formula. For this  the observation connecting    the distributions of $D^-_{H_\eta}$ and  $M_{\theta_\delta}$ (see (\ref{DM1})) is crucial. It is seen that the approach yields a very appealing formula characterizing the distribution where also the function appearing in Lehoczky's formula is present.

In addition to the papers refered above there are many others serving both theoretical and applied interests. For alternative proofs for Taylor's theorem we refer to Williams \cite{williams76} where the theory of local time is used, and Salminen and Yor \cite{salminenyor03b} where an approach based on the Kennedy martingale, cf. \cite{kennedy76}, is applied. See also Meilijson  \cite{meilijson03} and Hu, Zhi and Yor \cite{hushiyor15}. Drawdowns and related objects for L\' evy processes have also been much studied. We refer to Mijatovic and Pistorius \cite{mijatovicpistorius12} and a more recent contribution Mayerhofer \cite{mayerhofer2019}; also for further  references.

The application indicated in Taylor's and Lehoczky's papers concerns a selling strategy on a stock market where the asset should be sold when the price has fallen below the previous maximum by a given fixed amount. Much of the more recent litterature on drawdowns (and drawups) finds its applications in financial mathematics; in particular, analyzing the properties of the options based on drawdowns. We refer, e.g.,  to   \cite{pospisilvecerhadji09}, \cite{CNP12}, \cite{zhangh15}, \cite{gapeevrodosthenous16}, \cite{dassioslim19}, \cite{zhangli23} together with further references in these papers. For the connection with the change point detection problem, see  \cite{pospisilvecerhadji09}.

Concerning the maximum  drawdown before a hitting time the potential applications discussed in Brockwell's and Malyutin's papers are in queuing theory and in mathematical biology when modeling  cell growth. Later also this functional is much studied in the framework of financial mathematics, see, e.g.,   \cite{douady00}, \cite{magdonismail04}, \cite{salminenvallois07},   \cite{CNP12} and the references therein.

In addition to the topics listed above we discuss the process   $(M_{\theta_\delta})_{\de\geq 0}$.  It is seen that this process is Markov and many characteristics of it are presented. In    \cite{salminenvallois20}  the process  $(D^-_{H_\eta})_{\eta\geq 0}$ is similarly analyzed. As indicated above,  the one-dimensional marginals of these processes are closely related. Hence, a natural task is to compare the probabilistic structures of the processes more closely. In particular, we deduce  the Markov generator of $(D^-_{H_\eta})_{\eta\geq 0}$  from the generator of   $(M_{\theta_\delta})_{\de\geq 0}$.

The paper is organized as follows. In the next section some preliminary facts, needed later in the paper, on linear diffusions are presented.
Two classes of diffusions are introduced for which the results are derived. In particular, geometric Brownian motion when the parameters are such that the process does not tend to 0 is included. Taylor's and Lehoczky's formulas are given in Section 3.  Lehoczky's formula in  \cite{lehoczky77} does not, in fact,  cover the case when the diffusion is defined, e.g., on $(0,\infty)$. In Theorem \ref{prop_L} we have extended the formula accordingly.  The connection with  $M_{\theta_\delta}$ and $D^-_{H_\eta}$ is shown in
Section 4 where also the joint distribution of  ${H_\eta}$ and  $D^-_{H_\eta}$ is characterized in the extended form as explained above.  In Section 5 Lehoczky's and Malyutin's formulae (in the extended form) are derived using the excursion theory of one-dimensional diffusions. The process   $(M_{\theta_\delta})_{\de\geq 0}$ is analyzed in Section 6. 


\section{Preliminaries}
\label{sec1}
 Let $X=(X_t)_{t\geq 0}$ be a regular one-dimensional diffusion in the sense of  It\^o and McKean \cite{itomckean74}, see also \cite{borodinsalminen15}.  In particular, $X$ is a strong Markov process with continuous sample paths taking values on an interval  $I\subseteq \R$.
We let  $\P_x $ and $\E_x $ denote the probability measure and
the expectation operator associated with  $X$ when $X_0=x$. The standard notation $\cF_t$ is used for the
$\sigma$-algebra generated by $X$ up
to time $t\geq 0,$ and we set $\cF:=\cF_\infty.$ It is also assumed that $X$ is not killed inside $I$.

Then it is known from the theory of one-dimensional diffusions that
for $\alpha >0$
\begin{equation}
\label{for1}
\E_x(\exp(-\alpha\,H_y))=
\begin{cases}
{\displaystyle{\frac{\psi_\alpha(x)}{\psi_\alpha(y)}}},& x\leq y,\\
{\displaystyle{\frac{\varphi_\alpha(x)}{\varphi_\alpha(y)}}},& x\geq y,\\
\end{cases}
\end{equation}
where   $H_y$ is the first hitting time of $y$ (see (\ref{Hit})) and  $\varphi_\alpha$  ($\psi_\alpha$) are a decreasing (increasing) solution of the generalized differential equation
\begin{equation}
\label{for2}
\frac{d}{dm}\frac{d}{dS} u =\alpha u,
\end{equation}
associated with $X$. The notation $m$ and $S$ are used for the speed measure and the scale function, respectively, of $X.$
Imposing appropriate boundary conditions (see, e.g.,   \cite{borodinsalminen15})
determine $\psi_\alpha$ and $\varphi_\alpha$ uniquely up to  multiplicative constants.  The Wronskian $w_{\alpha}$ is defined as
\begin{align}
\label{wronskian}
\nonumber
w_{\alpha}&:=\psi^{+}_{\alpha}(x)\varphi_\alpha(x)-\psi_\alpha(x)\varphi^{+}_{\alpha}(x)\\
&\hskip2mm =\psi^{-}_{\alpha}(x)\varphi_\alpha(x)-\psi_\alpha(x)\varphi^{-}_{\alpha}(x),
\end{align}
where the superscripts $^+$ and $^-$ denote the right and left
derivatives with respect to the scale function. Notice that $w_\alpha$ does not depend on $x$.
It is well-known (see \cite{itomckean74} p. 150) that
\begin{equation}\label{greenkernel}
g_\alpha(x,y):=
\left\{
\begin{array}{rl}
w^{-1}_\alpha\psi_\alpha(x)\varphi_\alpha(y),\quad x\leq y,\\
w^{-1}_\alpha\psi_\alpha(y)\varphi_\alpha(x),\quad x\geq y,
\end{array}\right.
\end{equation}
serves as a resolvent density (also called the Green function) of $X,$ i.e., for any Borel subset $A$ of $I$
\begin{equation}
\label{RES1}
G_\alpha(x,A):=\E_x\left(\int_0^\zeta {\rm e}^{-\alpha t}{\bf 1}_A(X_t)\, dt\right) = \int_A g_\alpha(x,y)\,m(dy),
\end{equation}
where $\zeta$ denotes the lifetime of the diffusion. 
 
For simplicity, it is assumed throughout the paper  that the scale function $S$ is continuously differentiable and the speed measure $m$ has a continuous derivative with respect to the Lebesgue measure.  Under these assumptions,   $\varphi_\alpha$  and $\psi_\alpha$ are continuously differentiable. However, at many places in this paper,  it is practical to consider derivatives with respect to the scale function. Hence, we keep the notations for the derivatives introduced above; e.g.,
$$
\psi^{+}_{\alpha}(x)=\psi^{-}_{\alpha}(x)=\frac{d\psi_\alpha(x)}{dx}\frac{dx}{dS(x)}
= \frac{1}{S'(x)}\,\psi^{'}_{\alpha}(x).
$$

Let $l$ and $r$ denote the left hand and the right hand, respectively, end point of $I$. To fix idea, we let, throughout the whole paper, $l\geq -\infty$ and $r=+\infty$. Moreover, we concentrate ourselves to two specific classes of underlying diffusions:
\medskip

\noindent
Class 1.   $X$ is recurrent with a) $S(+\infty)=+\infty, S(l)=-\infty$ or b)
 $S(+\infty)=+\infty$, $S(l)>-\infty$, $l>-\infty$ and reflecting.
\medskip

\noindent
Class 2.    $X$ is transient with a) $S(+\infty)<\infty$, $S(l)=-\infty$, or b) $S(+\infty)<\infty$,  $S(l)>-\infty$, $l>-\infty$ and reflecting.
\medskip

\noindent
In particular, standard Brownian motion belongs to Class 1a and  Brownian motion with drift $\mu>0$ to Class 2a. Notice, e.g., that  killed BM on $(-\infty,0)$ (killed when hitting 0) is excluded. For a geometric Brownian motion $(X_t)_{t\geq 0}$ with
$$
X_t :=x\exp\lp\lp\mu-\sigma^2/2\rp t+\sigma W_t\rp,
$$
where $(W_t)_{t\geq 0}$ is a standard BM, $x>0$, and $\mu,\sigma\in\R$, it follows  by checking the scale function (cf. \cite{borodinsalminen15} p. 136) that $X$ is in Class 1a  in case $\mu=\sigma^2/2$, and  in Class 2a  if $\mu>\sigma^2/2$.


\section{Main theorems}
\label{main}

\subsection{Drawdown time}
\label{sec2}

In this section we recall Lehoczky's formula  characterizing the joint distribution of the first drawdown time ${\theta_\delta}$ and the maximum $M_ {\theta_\delta}$ at this moment.  Assuming $\theta_\delta<\infty$ -- we show that this holds for diffusions in  class 1 -- then  $M_ {\theta_\delta}$ is bigger than the initial value of $X$ but  also that $M_ {\theta_\delta}$ is bigger than $\de+ l$ which is important to remember in case $l>-\infty$. Below (\ref{ei00}) is Lehoczky's formula. Taylor's formula for standard Brownian motion is taken up in Example \ref{BM}  and displayed in (\ref{BM1}). In Theorem \ref{prop_L} Lehoczky's formula is extended to include the possibility that $l>-\infty$. As explained in the introduction, Lehoczky's theorem is proved  in Section \ref{sec4} using excursion theory.  

\begin{theorem}
\label{prop_L}
Let   $X$ be in class 1 or 2.
  Then the joint Laplace transform of $M_{\theta_\delta}$ and $\theta_\delta$ is given for $x\in I$,   $\delta>0$,  $\alpha>0$ and $\beta\geq 0$  by
\begin{equation}\label{ei00}
    \E_x\left(\exp(-\alpha \theta_\delta -\beta M_{\theta_\delta})\right)=
    \frac{\psi_\alpha(x)}{\psi_\alpha(x\vee(\de+l))} \int_{x\vee (\delta +l)}^\infty c_\alpha(y;\delta)\qquad\qquad\qquad
  \end{equation}
  \[\qquad\qquad\qquad \times  \displaystyle \exp\lp- \beta y-\int_{x\vee(\de+l)}^y b_\alpha(z;\de) \,dS(z)\rp dS(y)\]
where $\vee$ is the usual maximum operator, 
\begin{equation}
\label{B}
    b_\alpha(y;\delta):= \frac{\varphi_\alpha(y-\delta)\psi^-_\alpha(y)-\varphi^-_\alpha(y)\psi_\alpha(y-\delta)}
    {\varphi_\alpha(y-\delta)\psi_\alpha(y)-\varphi_\alpha(y)\psi_\alpha(y-\delta)}
\end{equation}
and
\begin{align}
\label{C}
    \nonumber
    c_\alpha(y;\delta)&:= \frac{\varphi_\alpha(y)\psi^-_\alpha(y)-\varphi^-_\alpha(y)\psi_\alpha(y)}
    {\varphi_\alpha(y-\delta)\psi_\alpha(y)-\varphi_\alpha(y)\psi_\alpha(y-\delta)}\\
    &
 =\frac{w_\alpha}
       {\varphi_\alpha(y-\delta)\psi_\alpha(y)-\varphi_\alpha(y)\psi_\alpha(y-\delta)}.
\end{align}
Moreover, for $y\geq x\vee (\de +l)$ 
\begin{equation}
\label{eq:ei01}
\P_x(M_{\theta_\delta}>y)=\exp\l(-\int_{x\vee(\de+l)}^y\frac{dS(z)}{S(z)-S(z-\delta)}\r).
\end{equation}
\end{theorem}

\begin{remark}
\label{rmk0}
{\bf 1.}
In (\ref{ei00})  the convention $\exp(-\infty)=0$ is used in case  $\theta_\delta=+\infty$. Clearly, if  $\theta_\delta<+\infty$ then also $M_{\theta_\delta}<+\infty$. 

\noindent
{\bf 2.} For a diffusion in class 2 the formula (\ref{ei00}) is valid also for $\alpha=0$ with $\psi_0\equiv 1$ and $\varphi_0(x)=S(+\infty)-S(x)$.  If the diffusion is  in class 1 approriate functions $\psi_0$ and $\varphi_0$ do not exist. To find the Laplace transform of $M_{\theta_\delta}$ one could try to passage to limit as $\alpha\to 0$. However, in the proof in Section \ref{sec4} we do not exploit this approach but instead work directly with $\alpha=0$.

\noindent
{\bf 3.} From (\ref{ei00}) it immediately follows that for the bounded and measurable function $f$ we have
 \begin{align*}
\nonumber
&\E_x\left(\e^{-\alpha \theta_\delta}\, f(M_{\theta_\delta})\right)
=
\int_{x}^\infty f(y)\exp\l(-\int_{x}^yb_\alpha(z;\delta )dS(z)\r)\, c_\alpha(y;\delta) dS(y).
\end{align*}

\end{remark}

We conclude this subsection with a result  stating  the finiteness of ${\th_\de}$ in Class 1. We demonstrate via examples that  it is possible that $X$ converges to $+\infty$ so fast that  $\theta_\de=+\infty $  with positive probability for all values on $\de$, and that  there are diffusions in Class 2 such that a.s. $\theta_\de<+\infty $ for all values on $\de$. 

\noindent
\begin{proposition}
\label{finiteness1}
Let   $X$ be in Class 1. Then  $\th_\de$ is finite for all $\de>0$ and so is $M_{\th_\de}$.
\end{proposition}
\begin{proof}
Consider the distribution of $M_{\th_\de}$ as given in  (\ref{eq:ei01}). For the integral on the right hand side we have since $S$ is non-decreasing
 \begin{align*}
\int_{x\vee(\de+l)}^y\frac{S'(z)dz}{S(z)-S(z-\delta)}&\ge\int_{x\vee(\de+l)}^y\frac{S'(z)dz}{S(z)-S(x-\delta)}\\
&\ge\ln\lp\frac{S(y)-S(x-\de)}{S({x\vee(\de+l)})-S(x-\delta)}\rp.
 \end{align*}
Since, by the assumption, $ \lim_{y\to\infty}S(y)=+\infty$ it follows that
$$
\P_0(M_{\th_\de}=+\infty)=0,
$$
which implies that  $\th_\de$ is finite for all $\de>0$, as claimed.
\end{proof}


\noindent

\begin{example}\label{ex1}
Assume $l=0$ and
\[S(z):=1-\exp\big(-\e^z\big),\quad z\geq 0.\]
Then $X$ is in Class 2 (case b)) and
\begin{equation}\label{S1}
 \P_x(\theta_\delta<+\infty)=\exp\Big(-\int_{\e^{x\vee \de}}^{+\infty} \frac{du}{\e^{(1-\rho)u}-1}\Big)<1
\end{equation}
where $\rho:=\e^{-\de}$. Indeed,
\[S(z)-S(z-\delta) =\exp\big(-\e^z\big)\Big(\exp\big((1-\rho)\e^z\big)-1\Big)\]
and
\[ \P_x(\theta_\delta<+\infty)=
\exp\Big(-\int_{x\vee \de}^{\infty}\frac{\e^z dz}{\exp\big((1-\rho)\e^z\big)-1}\Big) \]
Setting $u=\e^z$ yields \eqref{S1}.
\end{example}

\begin{example}\label{ex2}
Assume $l=0$ and $\displaystyle S(z):=1-\e^{-z},\, z\geq 0$. Reflecting Brownian motion with drift 1 is a particular example  satisfying these assumptions. 
Clearly,  $X$ is in Class 2 (case b). It straightforward to show (and left to the reader)  that   $\P_x(\theta_\delta<+\infty)=1$.  
\end{example}

\subsection{Maximum drawdown} 
\label{sec3}



Recall from (\ref{MXD}) the definition of $D^-_t$. Here we study this functional up to the first hitting time ${H_\eta}$, where $\eta$ is assumed to be bigger than the initial value of $X$. The distribution of
$D^-_{H_\eta}$ is given in (\ref{eqSaVa}) below. Brockwell derived in \cite{brockwell74}  this distribution  for Brownian motion with non-negative drift. Malyutin in \cite{malyutin87} gives a formula characterizing the joint distribution of ${H_\eta}$ and $D^-_{H_\eta}$  corresponding to (\ref{ei11}) below, but the r\^ole of $l$ is not transparent in Malyutin's work. The formula in \cite{malyutin87} is made explicit for  Brownian motion with drift, Ornstein-Uhlenbeck processes and Bessel processes. We remark also that in \cite{salminenvallois20}, see Proposition 2.3, the distribution of $D^-_{H_\eta}$  is derived using stochastic calculus. 

\begin{theorem}
\label{SaVa}
Assume  that the underlying diffusion is in class 1 or 2. Then ${H_\eta}$ is finite $\P_x$-a.s. for $x<\eta$, and the
distribution of  $D^-_{H_\eta}$  is given for $x<\eta$ and $0\leq y\leq \eta-l$ by
\begin{equation}
\label{eqSaVa}
  \P_x(D^-_{H_\eta}<y)=\exp\left(-\int_{x\vee(y+l)}^{\eta}\frac{S(dz)}{S(z)-S(z-y)}\right).
\end{equation}
\end{theorem}

A striking observation is the similarity of the right hand sides of the formulas (\ref{eqSaVa}) and (\ref{eq:ei01}). Indeed,  there is a close connection between $D^-_{H_\eta}$ and $M_{\theta_\delta}$ which can be deduced without any knowledge of the explicit forms of the distributions. To see this, note that
\begin{equation}\label{theta2}
  \theta_\delta=\inf\{t\geq 0\, :\,  D^-_t>\delta\},
\end{equation}
and then a.s.
\begin{equation}\label{DM1}
   \{ D^-_{H_\eta}\leq \delta\}=\{\theta_\delta\geq H_\eta\}=\{M_{\theta_\delta}\geq \eta\},
\end{equation}
where the first equality follows from (\ref{theta2}) and the second one from  the monotonicity of $t\mapsto M_t$. We may thus deduce from Lehoczky's formula (\ref{eq:ei01})  the distribution function of $D^-_{H_\eta}$, and vice versa.   


Next we characterize the joint distribution  of  $D^-_{H_\eta}$ and $H_\eta$. The proof is presented in Section 5 and follows from excursion theoretical calculations needed also for the proof of Lehozcky's formula. 


\begin{theorem}
\label{SaVa0}
The joint distribution  of $D^-_{H_\eta}$ and $H_\eta$ is determined  for $\eta >x\vee (y+l)$ by
\begin{align}
\label{eiNew}
\nonumber
\hskip-1cm\E_x\left(\exp(-\alpha H_\eta);\,   D^-_{H_\eta}<y\right)
&=
\frac{\psi_\alpha(x)}{\psi_\alpha(x\vee(y+l))}\\
&\hskip1cm\times\exp\lp-\int_{x\vee(y+l)}^\eta b_\alpha(z,y) S(dz)\rp,
\end{align}
where $\alpha>0$ and $b_\alpha$ is as in (\ref{B}).
\end{theorem}

\begin{remark}
\label{rmk00}
An alternative proof of (\ref{eiNew}) can be done via conditioning. To explain this briefly,  let $T$ be an exponentially distributed random variable with mean $1/\alpha$. Assume that $T$ is independent of $X$. Recall that $r=+\infty$  and consider for $l<x<\eta$ 
 \begin{align*}
\hskip-.5cm\E_x\left(\exp(-\alpha H_\eta);\,   D^-_{H_\eta}<y\right)
&=\P_x\left(H_\eta<T ,\,   D^-_{H_\eta}<y\right)
\\
&=\P_x\left(  D^-_{H_\eta}<y\,|\,H_\eta<T \right)\P_x\left( H_\eta<T \right).
\end{align*}
The conditional distribution above can be found by applying the theory of $h$-transforms (and (\ref{eqSaVa})). This yields
\begin{align}
\label{ei11}
\nonumber
&\E_x\left(\exp(-\alpha H_\eta);\,   D^-_{H_\eta}<y\right)
\\
&\hskip2cm
=
\frac{\psi_\alpha(x)}{\psi_\alpha(\eta)}
\exp\lp-\int_{x\vee(y+l)}^\eta c_\alpha(z,y)
\frac{\psi_\alpha(z-y)}{\psi_\alpha(z)} S(dz)\rp,
\end{align}
and it can be shown that (\ref{ei11}) and (\ref{eiNew}) are equivalent.\\
\end{remark}

\subsection{Additional remarks}
\label{ADD}

It is possible to develop the above concepts and results to various directions. Here we indicate  some of these generalizations.

\smallskip
\noindent
{\bf 1. More general drawdown time depending on the maximum.}
Let  $\phi$ be a continuous positive function  consider the stopping time
\begin{align*}
&
\theta_\phi := \inf\{t\geq  0\, :\,  M_t -X_t = \phi(M_t)\}.
\end{align*}
Note that if the function $\phi$ is constant and equal to $\de$, then $\theta_\phi$ coincides with the stopping time $\theta_\de$ defined by \eqref{dtheta}.
The joint distribution of $M_{\theta_\phi}$ and  $\theta_\phi$ is characterized in  \cite{lehoczky77}.
Recall (in case $l=-\infty$) therefrom
\begin{equation*}
\P_x(M_{\theta_\phi}>y)=\exp\l(-\int_{x}^y\frac{dS(z)}{S(z)-S(z-\phi(z))}\r).
\end{equation*}
It would be easy to adapt our proof of Theorem \ref{prop_L} to determine the distribution of the pair $\big( M_{\theta_\phi}, \theta_\phi\big)$.
The stopping time $\theta_\phi$ appears also in studies on Skorokhod embedding, see  Az\'ema and Yor \cite{azemayor79} and  \cite{azemayor79b}.

\smallskip
\noindent
{\bf 2. Maximum drawdown for killed diffusions.}
Let $g$ be a piecewise continuous non-negative (but non-vanishing)  function and define an additive functional associated with a diffusion $X$ via
 \begin{equation}
 \label{ADDI}
 A_g(t):=\int_0^t g(X_s)\, ds,
  \end{equation}
where $g$ is bounded and piecewise continuous. Let $T$ be an exponentially distributed random variable with mean $1$ and independent of  $X$. We assume that $X$  in class 1 or 2 and introduce for $t>0$
 \begin{equation}
 \label{KD}
 \widehat X_t:=
 \begin{cases}
 X_t,&  A_g(t)<T,\\
 \partial,&A_g(t)\geq T,\\
 \end{cases}
  \end{equation}
where $\partial$ is the so called  cemetary state isolated from the state space of $X$. As is well known,  $\widehat X=(\widehat X_t)_{t\geq 0}$ is a diffusion and it is said that $\widehat X$ is obtained from $X$ by killing via the additive functional defined in (\ref{ADDI}). Let $\varphi_{\alpha,g}$  and $\psi_{\alpha,g}$ be a decreasing and an increasing, respectively, solution of the generalized differential equation
\begin{equation}
\label{for2b}
\frac{d}{dm}\frac{d}{dS} u =(\alpha + g) u
\end{equation}
describing the resolvent density of $\widehat X$ similarly as explained in (\ref{for1}),  (\ref{for1}), (\ref{wronskian}), and (\ref{greenkernel}) for the (unkilled) diffusion $X.$
 Then letting $\widehat\E$ denote the expectation operator associated with  $\widehat X$ it can be proved
 \begin{align}
\label{ei11X}
\nonumber
&\widehat\E_x\left(\exp(-\alpha H_\eta);\,   D^-_{H_\eta}<u\right)\\
\nonumber
&\hskip2cm
=\E_x\left(\exp(-\alpha H_\eta- A_g(H_\eta) );\,   D^-_{H_\eta}<u\right)
\\
\nonumber
&\hskip2cm
=
\frac{\psi_{\alpha,g}(x)}{\psi_{\alpha,g}(\eta)} \exp\lp-\int_{x\vee(u+l)}^\eta c_{\alpha,g}(z,u)
\frac{\psi_{\alpha,g}(z-u)}{\psi_{\alpha,g}
(z)} S(dz)\rp
\\
&\hskip2cm
=\exp\lp-\int_{x\vee(u+l)}^\eta b_{\alpha,g}(z,u) S(dz)\rp,
\end{align}
where  $c_{\alpha,g}$ and  $b_{\alpha,g}$ are constructed as  $c_{\alpha}$ and $b_{\alpha}$ in (\ref{C}) and (\ref{B}), respectively, but with the functions  $\varphi_{\alpha,g}$  and $\psi_{\alpha,g}$.


\subsection{Examples}
\label{secX}

\begin{example}
\label{BM}
In our first example  $X=(X_t)_{t\geq 0}$ is a standard Brownian motion. The joint distribution of  $\theta_\delta$ and $M_{\theta_\delta}$ in this case was characterized by Taylor in \cite{taylor75} -- a few years before Lehoczky \cite{lehoczky77}. In fact, in \cite{taylor75} Brownian motion with drift is analyzed. For $X$ we have 
 $$
 S(y)=y,\ \varphi_\alpha(y)=\e^{-y\sqrt{2\alpha}},\ \psi_\alpha(y)=\e^{y\sqrt{2\alpha}}, \text{ and } w_\alpha=2\sqrt{2\alpha},
 $$
 and using these in Theorem \ref{prop_L} yields  formulas
 \begin{align}
\label{BM1}
&\E_x\left(\exp(-\alpha \theta_\delta -\beta M_{\theta_\delta})\right)
= \frac{\sqrt{2\alpha}\,\e^{-\beta x}}{\sqrt{2\alpha}\,\cosh(\delta\sqrt{2\alpha})+
\,\beta \sinh(\delta\sqrt{2\alpha})},
\end{align}
 \begin{align}
\label{BM11}
&\E_x\left(\exp(-\alpha \theta_\delta)\right)
= \frac{1}{\cosh(\delta\sqrt{2\alpha})},
\end{align}
and for $y\geq x $
\begin{equation}
\label{BM2}
\P_x(M_{\theta_\delta}>y)=\e^{- (y-x)/\delta},
\end{equation}
i.e.,  $(M_{\theta_\delta}-x)/\de$ is exponentially distributed.
 \end{example}

\begin{example}
\label{RBM}
Let $X=(X_t)_{t\geq 0}$ be a Brownian motion reflected at 0 and living in $[0,+\infty)$.
Since the scale function of $X$ is $S(z)=z$ we have from (\ref{eq:ei01}) for
 $y\geq x\vee\de$
 \begin{equation}\label{RBM1}
 \P_x\left(M_{\theta_\delta}>y\right)=  \exp\left(-\frac 1\de\left(y-x\vee\de\right)\right).
 \end{equation}
To analyze formula (\ref{ei00}) recall from \cite{borodinsalminen15} p. 124 that for $X$
$$
\psi_\alpha(y)=\cosh(y\sqrt{2\alpha}),\ \ \varphi_\alpha(y)=\e^{-y\sqrt{2\alpha}},\ \
w_\alpha=\sqrt{2\alpha}.
$$
Straightforward calculations show that the functions $b_\alpha(y;\de)$ and $c_\alpha(y;\de)$ defined in (\ref{B}) and (\ref{C}), respectively, do not depend on $y$ and are given by
$$
b_\alpha(y;\de)=\sqrt{2\alpha}\cosh(\de\sqrt{2\alpha})/\sinh(\de\sqrt{2\alpha}),
$$
and
$$
c_\alpha(y;\de)=\sqrt{2\alpha}/\sinh(\de\sqrt{2\alpha}).
$$
Hence, from Theorem \ref{prop_L}
 \begin{align*}
\E_x\big(\exp(-\alpha \theta_\delta &-\beta M_{\theta_\delta})\big)\\
&=\frac{\cosh(x\sqrt{2\alpha})}{\cosh((x\vee\de)\sqrt{2\alpha})} \frac{\sqrt{2\alpha}\,\e^{-\beta(x\vee\de)}}{\sqrt{2\alpha}\,\cosh(\delta\sqrt{2\alpha})+
\,\beta \sinh(\delta\sqrt{2\alpha})}.
\end{align*}
In particular, for $\beta=0$  and $x=0$ 
 \begin{equation}
\label{RBM11}
\E_0\left(\exp(-\alpha \theta_\delta )\right)
= \frac1{\cosh^2(\delta\sqrt{2\alpha})}.
\end{equation}
Recall that 
 \begin{equation*}
\E_0\left(\exp(-\alpha \sigma_\delta )\right)
= \frac1{\cosh(\delta\sqrt{2\alpha})},
\end{equation*} 
where $ 
\sigma_\delta:=\inf\{t: X(t)=\delta\}.
$
Hence, we have the relation  
$$ 
\theta_\delta\,{\stackrel {(d)}{=}}\,\sigma^{(1)}_\delta+\sigma^{(2)}_\delta,
$$
where $\sigma^{(1)}_\delta$ and $\sigma^{(2)}_\delta$ are independent,
$$
\sigma^{(1)}_\delta\,{\stackrel {(d)}{=}}\, \sigma^{(2)}_\delta\,{\stackrel {(d)}{=}}\,\sigma_\delta,
$$
and ${\stackrel {(d)}{=}}$ means that ``...is indentical in law with...''.  For the maximum drawdown we have from (\ref{eiNew}) for $x,y\in[0,\eta]$ 

\[
\E_x\left(\exp(-\alpha H_\eta);\, D^-_{H_\eta}<y\right)
=\frac{\cosh\big(x \sqrt{2\alpha}\big)}{\cosh\big((x\vee y) \sqrt{2\alpha}\big)} \qquad \qquad \qquad \qquad\]
\[\qquad \qquad \qquad \qquad \qquad \qquad \qquad \qquad \times \exp\lp-\frac{\sqrt{2\alpha}\cosh(y\sqrt{2\alpha})}{\sinh(y\sqrt{2\alpha})}(\eta-x\vee y)\rp.
\]
\end{example}


\section{Proofs of Theorems \ref{prop_L} and  \ref{SaVa0}} \label{sec4}

This section  is organized in 4 subsections. In the first one we present  the notation and the basic facts from the classical excursion theory for excursions from a point.   The second one concerns the theory for excursions below the maxima of the underlying diffusion.  A crucial element hereby is the master formula (\ref{master}) due to  Fitzsimmons \cite{fitzsimmons13}, see also \cite{pitmanyor03a}. These two subsections can be seen as background material to make the paper more readable. The main ingredients of the proof are given in Subsections \ref{sec33} and \ref{sTpL}. In particular, the functions  $b_\alpha$ and $c_\alpha$, cf. (\ref{B}) and (\ref{C}), respectively, are expressed   in terms of the excursion law, cf. \cite{fitzsimmons22}.  Finally, the proofs of Theorems \ref{prop_L} and
 \ref{SaVa0} are collected from the presented facts at the end of Subsection \ref{sTpL}.

\subsection{Excursions around a fixed point} \label{sec4a}
Let $( L^y_t\, :\, t \ge 0, y \in I)$ denote a jointly continuous version  of local time for $X$  normalized to be the occupation density relative to $m$. Then it holds
\begin{equation}
\label{LT}
\E_x\lp\int_0^\infty \e^{-\alpha t} dL^y_t\rp
 = g_\alpha(x,y)
\end{equation}
where $g_\alpha$ is the Green function as given in (\ref{greenkernel}).
Fixing a level $y\in I$, let $(T^y_s\, :\,  s \ge 0)$ be the right continuous inverse of $(L^y_t\, :\,  t \ge 0)$. 
Define the excursions
around $y$ via
\begin{equation}
\label{EXY}
\kappa^y_t(s):=\begin{cases}
 X_{T^y_{t-}+s}, & 0\le s<T^y_{t}- T^y_{t-}, \\
 y, & s\ge   T^y_{t}- T^y_{t-}.
\end{cases}
\end{equation}
 Then, $\kappa^y_t$ belongs to the space
\[U^y:=\big\{e \in C^y:
 \exists\  \zeta(e)<\infty \ {\text{such that}}\ e(t)\not=y\  {\text{for\ all}}\  0<t< \zeta(e)\]
\[{\text{and}}\ e(t)=y\ {\text{for\ all}}\  t\geq\zeta(e) \big\},\]
where $C^y$ denotes the class of continuous functions defined on $[0,+\infty)$ such that $e(0)=y$. Clearly,  $\zeta(e)=\inf\{t>0 ; e(t)=y\}$ for $e\in U^y$. Let $\cU^y$ denote the smallest $\sigma$-algebra in $U^y$ making all coordinate mappings
$t\mapsto e(t)$ measurable. Then $(U^y,\cU^y)$ is called an excursion space.
Notice that if $T^y_{t}- T^y_{t-}=0$ then $\kappa^y_t(s)=y$ for all $s\ge 0$, and this  is also taken to be an element in $U^y$ called $e^y$.  The It\^o excursion law $n^y$, which is a $\sigma$-finite non-negative measure defined on $(U^y,\cU^y)$, is determined by the identity
\begin{align}
\label{ITOO}
\nonumber
\E_x\Big(\sum_{t>0}Z_{T^y_{t-}}F(\kappa^y_t)\Big)&=\E_x\lp\int_0^\infty Z_{T^y_{t}} dt\rp n^{y}(F)\\
 &= \E_x\lp\int_0^\infty Z_{t} dL^y_t\rp n^{y}(F),
\end{align}
where $F$ is  a non-negative measurable function defined on $U^y$ such that $F(e^y)=0$, and $(Z_t)_{t\geq 0}$ is  a non-negative and progressively measurable process.
For (\ref{ITOO}), see \cite{revuzyor01} (1.10) Proposition p. 475 and (2.6) Proposition p. 483 where the latter reference is for Brownian motion but can be extended to our case. We refer also to Maisonneuve \cite{maisonneuve75}  (4.1) Theorem p. 401 where general strong Markov processes are considered.  

To give an explicit description of $n^y$, let $(Q^y_t)_{t\geq 0} $ denote the Markovian semigroup for the diffusion ${X^y}$ obtained from $X$ by killing at the hitting time $H_y$, that is, for $x\not= y$ and $A$ a Borel set in $(l,r)\setminus\{y\}$
\begin{equation}
\label{KILL}
\nonumber
Q^y_t(x,A):=\P_x\lp  X^y_t\in A\rp=\P_x\lp  X_t\in A,\, t<H_y\rp.
\end{equation}
It is well known that ${X^y}$ has a transition density $q^y$, say,  with respect to the speed measure. Hence it holds
\begin{equation}
\label{KILL20}
\nonumber
\P_x\lp X^y_t\in A\rp=\int_A q^y(t;x,z)\,m(dz).
\end{equation}
Recall  that $q^y(t;x,y)=0$ and that $q^y(t;\cdot,\cdot)$ is a symmetric function of $x$ and $z,$ i.e., $q^y(t;x,z) =q^y(t;z,x)$. Finally, let $f^H_{xy}$ denote the density of the $\P_x$-distribution of $H_y$. Now we are ready to characterize   the finite dimensional distributions of the It\^o excursion law $n^y$ for excursions below $y$.

\begin{theorem}
\label{ITO}
For $x_i<y, i=1,2,\dots,n,$ and $0<t_1<t_2<...<t_n$ it holds
\begin{align}
\label{FINIX}
\nonumber
&n^y\lp e(t_1)\in dx_1, e(t_2)\in dx_2, ... ,
e(t_n)\in dx_n\rp\\
 &\hskip2cm =\eta_{t_1}(dx_1) Q^y_{t_2-t_1}(x_1,dx_2)\cdot ...
 \cdot Q^y_{t_n-t_{n-1}}(x_{n-1},dx_n),
\end{align}
where $(\eta_t)_{t\geq 0}$ constitutes an entrance law for $(Q^y_t)_{t\geq 0} $, i.e., it satisfies
\begin{equation}
\label{ENT1X}
\nonumber
\int_{(l,r)\setminus y}\eta_t(dx) Q^y_s(x,A)
=\eta_{t+s}(A),
\end{equation}
and is given by
\begin{equation}
\label{ENT}
\eta_t(dx)=f^H_{xy}(t)m(dx).
\end{equation}
\end{theorem}
\begin{proof}
The result is extracted from \cite{rogerswilliams00II} No's 48-49 pp. 416-420, see also \cite{salminenvalloisyor07}. The only thing to clarify is the explicit formula (\ref{ENT}) for the entrance law. However, in ibid., we have the Laplace transform
\begin{equation}
\label{ENT2X}
\nonumber
\int_0^\infty \e^{-\alpha t}\eta_t(A)\,dt
=\frac{ G_\alpha(y,A)}{g_\alpha(y,y)}
\end{equation}
with $G_\alpha$ as given in (\ref{RES1}). Applying here the formulas (\ref{greenkernel}) and (\ref{for1}) yields
(\ref{ENT}).
\end{proof}

In the next proposition we give a useful way to determine the excursion law via a limiting procedure (cf. \cite{pitmanyor82} p. 437).  The construction is done for excursions below $y$, a similar formula is valid for excursions above $y$.
\begin{proposition}
\label{LIMIT}
Let $f_i, i=1,2,\dots,n,$  It holds for $x_i<y, i=1,2,\dots,n,$
\begin{align}
\label{LIMIT1}
\nonumber
&n^y\lp e(t_1)\in dx_1, e(t_2)\in dx_2, ... ,
e(t_n)\in dx_n\rp\\
 &\hskip.5cm =\lim_{\epsilon\to 0+}\frac{Q^y_{t_1}(y-\epsilon, dx_1)}{S(y)-S(y-\epsilon)}\, Q^y_{t_2-t_1}(x_1,dx_2)\cdot ...
 \cdot Q^y_{t_n-t_{n-1}}(x_{n-1},dx_n).
\end{align}
\end{proposition}

We need, in fact, an integrated form  of the formula (\ref{LIMIT1}) valid for functionals of excursions. For this aim, let
\begin{align*}
&C^*:=\big\{f: [0,\infty)\mapsto \R, \ {\text{continuous, and}}\
 \exists\  \zeta(f)<\infty\\
&\hskip4cm
\ {\text{such that}}\ f(t)=f(\zeta(f))\ {\text{for\ all}}\  t\geq\zeta(f) \big\}.
\end{align*}
In particular, notice that $U^y\subset C^*$ for all $y\in I$.
The next proposition can be proved from Proposition \ref{LIMIT} via monotone class-arguments.
\begin{proposition}
\label{LIMITX}
Let $F$ be a measurable, bounded and non-negative functional defined in $C^*$. Then
\begin{align}
\label{LIMIT12}
\nonumber
&\int_{U^y_-}F(e)n^y(de)=\lim_{\epsilon\to 0+}\frac{1}{S(y)-S(y-\epsilon)}\, {\E_{y-\epsilon}\lp F(X^y)\rp},
\end{align}
where ${U^y_-}\subset U^y $ denotes the excursions below $y$.
\end{proposition}

\subsection{Excursions below the maxima}\label{sec4b}

We proceed to study the excursion process  below the maxima (cf. also Fitzsimmons \cite{fitzsimmons13} and Pitman and Yor \cite{pitmanyor03a}).   Assume that $X_0=x$ and define for $y>x$ the first passage time over the level $y$ via
\begin{equation}\label{Hit+}
   H_{y+}:=\inf\{t\geq 0\, :\,X_t>y\}.
\end{equation}
 The excursions below the maxima are defined as follows. For $y>x$ assume first that $H_{y+}-H_y>0$ then we put
\begin{equation}\label{EXC}
\xi^y(t):=
\begin{cases}
X_{H_y+t}, &\quad 0\leq t\leq H_{y+}-H_y \\
y,  &\quad t>  H_{y+}-H_y.
\end{cases}
\end{equation}
In case  $H_{y+}-H_y=0$ we define $\xi^y(t):=y$ for all $t\geq 0$. Notice that $\xi^y\in U_-^y$. 
 The process $\Xi_x=((y,\xi^y)_{y\geq x})$ is a Poisson point process taking values in $[x,r)\times U_-^y$.  For a definition of a point process, see, e.g. section XII 1 in \cite{revuzyor01}. For Poisson point processes $\Xi_x$ we have the master formula (see \cite{fitzsimmons13})
\begin{align}
&\label{master}
\hskip-1cm
 \E_x\Big(\sum_{x\le y}Z_{H_y}F(\xi^y){\bf 1}_{\{H_y<+\infty\}}\Big)\\
 &\hskip3cm
 \nonumber
 =\E_x\Big(\int_x^\infty Z_{H_y} n^{y}_-(F){\bf 1}_{\{H_y<+\infty\}} dS(y)\Big),
\end{align}
where
$F$ is  a measurable non-negative  functional defined on $ U_-^y$ such that $F(e^y)=0$ for constant function $e^y$, $(Z_t)_{t\geq 0}$ is a non-negative and progressively measurable process
and $n^{y}_-$ is the restriction of $n^y$ on the set of excursions below $y$. From  (\ref{master}) it is seen that the intensity measure of $\Xi_x$ is given by $(dy,de)\mapsto dS(y)\, n^y_-(de)$ for $y\geq x$ and $e\in U_-^y$.

To specialize the master formula (\ref{master}) to our particular case involving $\theta_\delta$ and $M_{\theta_\delta}$ we recall (cf. (\ref{theta2})) that a.s. for all $y$ and $\delta $
\begin{equation}
\label{SETX}
\{y\leq M_{\theta_\delta}\}= \{H_y\leq \theta_\delta\}.
\end{equation}
Moreover,
\begin{equation}
\label{SETX1}
 \{H_y\leq \theta_\delta\}= \{H_y< \theta_\delta\}.
\end{equation}
To see (\ref{SETX1}) notice that if $H_y= \theta_\delta$ then also $X_{H_y}= X_{\theta_\delta}$ and $M_{H_y}= M_{\theta_\delta}$ which leads to a contradiction.  Letting $Z$ be a non-negative and progressively measurable process  then,  since $\theta_\delta$ is a stopping time, also the process defined by
\begin{equation}
\nonumber
t\mapsto Z'_t:=Z_t\,{\bf 1}_{\{t<\theta_\delta\}}1_{\{M_t>\de +l\}}
\end{equation}
is non-negative and progressively measurable. Since 
\[Z'_{H_y}=Z_{H_y} \,{\bf 1}_{\{y\leq M_{\theta_\delta}\}}1_{\{y>\de +l\}}\]
we obtain from \eqref{master} when applied for $Z'_{H_y}$ the identity 
\begin{align}
&
\hskip-1cm\label{master2}
 \E_x\Big(\sum_{{x\vee(\de+l)}< y}Z_{H_y}F(\xi^y){\bf 1}_{\{y\leq M_{\theta_\delta},
H_y<+\infty\}}\Big)\\
&\hskip3cm
\nonumber
=\E_x\Big(\int_{x\vee(\de+l)}^\infty Z_{H_y} n^{y}_-(F)
{\bf 1}_{\{y\leq M_{\theta_\delta},
H_y<+\infty\}} dS(y)\Big),
\end{align}
which is the key to the proofs of Theorem \ref{prop_L} and \ref{SaVa0}.

\subsection{Two crucial formulas}\label{sec33}

In the next lemma, the  functions $b_\alpha$ and $c_\alpha$, cf. (\ref{B}) and (\ref{C}) are expressed   in terms of the excursion law $n^y$, cf. Fitzsimmons \cite{fitzsimmons22}. These  expressions are used in the  next section.

\begin{lemma}
\label{lemma34}
It holds for $\delta>0$ and $y>\de+l$
\begin{align}
\label{FEB}
        &n^y_-\lp 1-\e^{-\alpha \zeta(e)}\,{\bf 1}_{\{ H_{y-\delta}(e)=\infty\}}\rp =
\begin{cases}
b_\alpha(y;\delta)& \alpha>0,\\
(S(y)-S(y-\de))^{-1},& \alpha=0,\\
\end{cases}
\end{align}

and
\begin{align}
\label{FEC}
&    n^y_-\lp \e^{-\alpha H_{y-\delta}(e)}\,;\, H_{y-\delta}(e)<\infty\rp
=
\begin{cases}
c_\alpha(y;\delta),& \alpha>0,\\
(S(y)-S(y-\de))^{-1},& \alpha=0,\\
\end{cases}
\end{align}
where $H_{y-\delta}(e)$ denotes the first hitting time of ${y-\delta}$ for a generic excursion~$e$.
\end{lemma}
\begin{proof}
For $0<\varepsilon<\delta$ recall the formulas, see \cite{itomckean74} p. 29, where the formulas for Brownian motion are discussed, see also Darling and Siegert \cite{darlingsiegert53},
\begin{align}
\label{HIT1}
  \E_{y-\varepsilon}\lp\e^{-\alpha H_{y-\delta}}\,;\, H_{y-\delta}<H_y\rp    =\frac{\varphi_\alpha(y-\varepsilon)\psi_\alpha(y)-\varphi_\alpha(y)\psi_\alpha(y-\varepsilon)}
    {\varphi_\alpha(y-\delta)\psi_\alpha(y)-\varphi_\alpha(y)\psi_\alpha(y-\delta)}
\end{align}
and
\begin{align}
\label{HIT2}
\nonumber
  &\hskip-1cm
  \E_{y-\varepsilon}\lp\e^{-\alpha H_{y}}\,;\, H_{y}<H_{y-\delta}\rp
  \\
  &\hskip2cm
  =\frac{\varphi_\alpha(y-\varepsilon)\psi_\alpha(y-\delta)-\varphi_\alpha(y-\delta)\psi_\alpha(y-\varepsilon)}
    {\varphi_\alpha(y)\psi_\alpha(y-\delta)-\varphi_\alpha(y-\delta)\psi_\alpha(y)},
\end{align}
where $y-\varepsilon>l$.
We prove first (\ref{FEC}). This is a  straightforward application of Proposition \ref{LIMITX}
\begin{align}
\label{FEC1}
\nonumber
   &  n^y\lp \e^{-\alpha H_{y-\delta}(e)}\,;\, H_{y-\delta}(e)<\infty\rp \\
   &\hskip 2cm
   = \lim_{\varepsilon \downarrow 0}\frac 1{S(y)-S(y-\varepsilon)}
   \E_{y-\varepsilon}\lp\e^{-\alpha H_{y-\delta}}\,;\, H_{y-\delta}<H_y\rp.
\end{align}
The fact that the functions $\varphi_\alpha$ and $\psi_\alpha$ have the scale derivatives yields the claimed formula. The claim (\ref{FEB}) is proved using the same approach. For this, consider first
\begin{align*}
 & \E_{y-\varepsilon}\lp1-\e^{-\alpha H_{y}}\,{\bf 1}_{\{ H_{y-\delta}>H_y\}}\rp
\\
&\hskip1cm
=
1
-  \frac{\varphi_\alpha(y-\varepsilon)\psi_\alpha(y-\delta)-\varphi_\alpha(y-\delta)\psi_\alpha(y-\varepsilon)}
    {\varphi_\alpha(y)\psi_\alpha(y-\delta)-\varphi_\alpha(y-\delta)\psi_\alpha(y)},
\end{align*}
and applying again Proposition \ref{LIMITX} gives the formula (\ref{FEB}); we skip the details. The formulas (\ref{FEC})  and  (\ref{FEC}) in case $\alpha=0$ are obtained by applying the well known fact that, e.g.,
\begin{equation*}
 \P_{y-\varepsilon}\lp H_{y-\delta}< H_y\rp =\frac{S(y)-S(y-\varepsilon)}{S(y)-S(y-\delta) }.
\end{equation*}
\end{proof}

\subsection{Completion of the proof}
\label{sTpL}

We begin with an identity which allows us to use the formalism of the excursion theory.

\begin{lemma}\label{lT1} The following identity holds for all excursions $\xi^y,\, y>l,$ below the maxima
\begin{equation}
\label{FF0}
{\bf 1}_{\{M_{\theta_\delta}<z\}}
=\sum_{x\vee(\de+l)<y<z}{\bf 1}_{\{y\leq M_{\theta_\delta}, H_{y-\delta}(\xi^y)<\infty\}}.
\end{equation}
\end{lemma}
\begin{proof} 
We can take  $z>{x\vee(\de+l)}$.  Assume that $M_{\theta_\delta}=y^*$. Then, if $y^*\geq z$ both sides of (\ref{FF0}) equal 0. For  $y^*< z$ the indicator on the right hand side is 0 for all $y>y^*$. Also for $y<y^*$  the indicator is 0 since for such $y$ the excursion $\xi^y$  does not hit $y-\delta$. For $y=y^*$ the indicator is 1, and (\ref{FF0}) holds also for $y^*< z$.
\end{proof}

\begin{lemma}\label{lT10} For $\alpha\geq 0$
\begin{equation}
\label{LAMB1}
\E_x\Big(\,\e^{-\alpha \theta_\delta}\,{\bf 1}_{\{M_{\theta_\delta}<z\}}
\Big)=\int_{x\vee(\de+l)}^z \lambda(y;x) n^y_-\big(\
\e^{-\alpha H_{y-\delta}}\,
 {\bf 1}_{\{H_{y-\delta}<\infty\}} \big) dS(y)
\end{equation}
where for $y>{x\vee(\de+l)}$
\begin{equation}
\label{LAMBDA}
\lambda(y;x):=\E_x\Big( \e^{-\alpha H_y}\, {\bf 1}_{\{y\leq M_{\theta_\delta}\}} \Big).
\end{equation}
\end{lemma}
\begin{proof} Using (\ref{FF0})
\begin{align}
\label{FF1X}
\nonumber
\hskip-.5cm
\e^{-\alpha \theta_\delta}\,{\bf 1}_{\{M_{\theta_\delta}<z\}}
&
=\sum_{x\vee(\de+l)<y<z}\e^{-\alpha \theta_\delta}\,{\bf 1}_{\{y\leq M_{\theta_\delta}, H_{y-\delta}(\xi^y)<\infty\}}\\
&
\nonumber
=\sum_{x\vee(\de+l)<y<z}\e^{-\alpha \theta_\delta}\,{\bf 1}_{\{H_y\leq {\theta_\delta}, H_{y-\delta}(\xi^y)<\infty\}}\\
&
=\sum_{x\vee(\de+l)<y<z}\e^{-\alpha H_y}\,{\bf 1}_{\{H_y\leq {\theta_\delta}\}}\, \e^{-\alpha H_{y-\delta}(\xi^y)}\,{\bf 1}_{\{H_{y-\delta}(\xi^y)<\infty\}},
\end{align}
where in the second step we have applied (\ref{SETX})  and for the third step notice that
\begin{equation}
\label{FF12}
{\th_\de}=H_y+ H_{y-\delta}(\xi^y),
\end{equation}
when $M_{\theta_\delta}=y$. Hence we conclude
\begin{align}
\label{FF2X}
\nonumber
\hskip-.5cm
\e^{-\alpha \theta_\delta}\,{\bf 1}_{\{M_{\theta_\delta}<z\}}
&=\sum_{x\vee(\de+l)<y<z}\e^{-\alpha H_y}\,{\bf 1}_{\{y\leq M_{\theta_\delta},H_y< \infty\}}\, \e^{-\alpha H_{y-\delta}(\xi^y)}\,{\bf 1}_{\{H_{y-\delta}(\xi^y)<\infty\}},
\end{align}
and can then apply directly  (\ref{master2}) to prove the claim.
\end{proof}

We proceed now with the calculation of the function $\lambda$ as defined in (\ref{LAMBDA}).
\begin{lemma}
\label{lT2}
For $y>x\vee(\de+l)$

\begin{equation}\label{LAMBDA1}
\lambda(y;x)=\frac{\psi_\alpha(x)}{\psi_\alpha(x\vee(\de+l))}
\exp\lp-\int_{x\vee(\de+l)}^y b_\alpha(z;\de) \,dS(z)\rp,
\end{equation}

where $b_\alpha$ is given in (\ref{B}) (see also (\ref{FEB})).
\end{lemma}

\begin{proof}
Let $y>x\vee(\de+l)$ and recall that $X_0=x$. Note first
\begin{equation}
\label{X1}
\nonumber
{\bf 1}_{\{y\leq M_{\theta_\delta}\}}=\prod_{x\vee(\de+l)<z<y}{\bf 1}_{\{H_{z-\delta}(\xi^z)=\infty\}},
\end{equation}
and
\begin{equation}
\label{X2}
\nonumber
H_y= H_{x\vee(\de+l)}+\sum_{x\vee(\de+l)<z<y}\lp H_{z+}-H_z \rp.
\end{equation}
Then
\begin{equation}
\label{X3}
\e^{-\alpha H_y}\,{\bf 1}_{\{y\leq M_{\theta_\delta}\}}=e^{-\alpha H_{x\vee(\de+l)}}
\prod_{x\vee(\de+l)<z< y}\e^{-\alpha\lp H_{z+}-H_z\rp}{\bf 1}_{\{H_{z-\delta}(\xi^z)=\infty\}}.
\end{equation}
Recall the following elementary fact: given a sequence $a_i, i=1,2,\cdots,$ such that $a_i\in[0,1] $ for all $i$ then with $a_0:=1$
\begin{equation}
\label{X4}
1-\prod_{k=1}^\infty a_k=\sum_{k=1}^\infty\lp\prod_{i=0}^{k-1}a_i\rp(1-a_k).
\end{equation}
We apply (\ref{X4}) for the countable product in (\ref{X3}) to obtain
\begin{align}
\label{X5}
&
\nonumber
\hskip-1cm
A:=e^{-\alpha H_{x\vee(\de+l)}}-\e^{-\alpha H_y}\,{\bf 1}_{\{y\leq M_{\theta_\delta}\}}\\
&\hskip-.5cm
=e^{-\alpha H_{x\vee(\de+l)}}\sum_{x\vee(\de+l)<z<y}\lp\prod_{x\vee(\de+l)<u<z}\e^{-\alpha\lp H_{u+}-H_u\rp}{\bf 1}_{\{H_{z-\delta}(\xi^z)=\infty\}}\rp\\
&
\nonumber
\hskip4cm\times \lp 1-\e^{-\alpha\lp H_{z+}-H_z\rp}{\bf 1}_{\{H_{z-\delta}(\xi^z)=\infty\}}\rp.
\end{align}
Using (\ref{X3}) for the product in (\ref{X5}) yields
\begin{align}
\label{X6}
&
A=e^{-\alpha H_{x\vee(\de+l)}}\sum_{x\vee(\de+l)<z<y}\e^{-\alpha H_z}\,{\bf 1}_{\{z\leq M_{\theta_\delta}\}}
\lp 1-\e^{-\alpha\lp H_{z+}-H_z\rp}{\bf 1}_{\{H_{z-\delta}(\xi^z)=\infty\}}\rp.
\end{align}
Applying (\ref{master2}) gives
\begin{align*}
&
\E_x\lp A\rp=\int_{x\vee(\de+l)}^y\E_x\lp e^{-\alpha H_{x\vee(\de+l)}}\e^{-\alpha H_z}\,{\bf 1}_{\{z\leq M_{\theta_\delta}\}}\rp
\\
&\hskip5cm
\times
n^{z}_-\lp 1-\e^{-\alpha\zeta(\xi^z)}{\bf 1}_{\{H_{z-\delta}(\xi^z)=\infty}\rp\,dS(z)\\
&\hskip1cm
=\E_x\Big(e^{-\alpha H_{x\vee(\de+l)}}\Big)\int_{x\vee(\de+l)}^ y\E_{x\vee(\de+l)}\lp \e^{-\alpha H_z}\,{\bf 1}_{\{z\leq M_{\theta_\delta}\}}\rp
b_\alpha(z;\de)\,dS(z),
\end{align*}
where Lemma \ref{lemma34} and the strong Markov property at time
$H_{x\vee(\de+l)}$ are used. 
Recall (cf. (\ref{for1})) also that
\begin{equation*}
\lambda\big(x\vee(\de+l);x\big)= \E_x\left(\exp(-\alpha H_{x\vee(\de+l)})\right) =\frac{\psi_\alpha(x)}{\psi_\alpha(x\vee(\de+l))}.
 \end{equation*}
 Therefore, since
 \[\E_x(A)=\lambda\big(x\vee(\delta+l);x\big)-\lambda(y;x),\]
 it follows that the function $\lambda$ satisfies
\begin{align}
\label{X91}
&\nonumber
\lambda\big(x\vee(\de+l);x\big)-\lambda(y;x)=\lambda\big(x\vee(\de+l);x\big)
\\
&\hskip4cm
\times\int_{x\vee(\de+l)}^ yb_\alpha(z;\de) \,\lambda(z;x\vee(\de+l))\,dS(z).
\end{align}
Assume first that $x\geq \delta+l$. Then (\ref{X91}) takes the form  
\begin{align}
\label{X91A}
&
1-\lambda(y;x)=\int_{x}^ yb_\alpha(z;\de) \,\lambda(z;x)\,dS(z),
\end{align}
and then
\begin{align}
\label{X91Y}
&
\lambda(y;x)=\exp\lp-\int_{x}^y b_\alpha(z;\de) \,dS(z)\rp.
\end{align}
When $x\leq \delta+l$ we have
\begin{align}
\label{X91X}
&
\lambda\big(\de+l;x\big)-\lambda(y;x)=\lambda\big(\de+l;x\big)\int_{\de+l}^ yb_\alpha(z;\de) \,\lambda(z;\de+l)\,dS(z).
\end{align}The integral in (\ref{X91X}) has, in fact,   been evaluated in (\ref{X91A}) (put therein $x=\de+l$). Using this  yields 
 for $x\leq \delta+l$
\begin{align}
\label{X91Z}
&
\lambda(y;x)=\frac{\psi_\alpha(x)}{\psi_\alpha(\de+l)}
\exp\lp-\int_{\de+l}^y b_\alpha(z;\de) \,dS(z)\rp.
\end{align}
This ends the proof  of the lemma.  
\end{proof}

\noindent
We finalize now the proofs of Theorems \ref{prop_L} and \ref{SaVa0}. Firstly, Lehoczky's formula (\ref{ei00}) follows  from Lemma \ref{lT10} using therein  (\ref{FEC}) and (\ref{LAMBDA1}). To find  the distribution of $M_{\theta_\delta}$ (cf.  Remark \ref{rmk0}.2), i.e., to prove (\ref{eq:ei01}), put $\alpha=0$ in (\ref{LAMB1}) to have
\begin{equation}
\label{LAMB2}
\phi(z)=\int_{x\vee(\de+l)}^z \lp 1-\phi(y)\rp n^y_-\lp H_{y-\delta}<\infty\rp dS(y)
\end{equation}
where 
\begin{equation}
\label{LAMBDA2}
\phi(z):=\P_x\lp M_{\theta_\delta}< z \rp.
\end{equation}
Notice that $\phi \big(x\vee (\delta +l)\big)=0$. Using (\ref{FEC}) and differentiating in (\ref{LAMB2}) yield for $z>l+\de$
\begin{equation}
\label{LAMBDA3}
\frac{\phi'(z)}{1-\phi(z)}=\frac{S'(z)}{S(z)-S(z-\delta)}.
\end{equation}
Integrating in  (\ref{LAMBDA3}) over the interval $(x\vee(\de+l),y)$ leads to  (\ref{eq:ei01}).  Secondly, consider the statement in Theorem \ref{SaVa0}.  Applying (\ref{DM1}) we have
\begin{equation*}
\E_x\lp\exp(-\alpha H_y); D^-_{H_y}<\delta\rp=\E_x\Big(\exp(-\alpha H_y); M_{\theta_\delta}>y\Big)=\lambda(y;x),
\end{equation*}
where $\lambda$ is as in (\ref{LAMBDA}). The claim in Theorem \ref{SaVa0} follows now from Lemma \ref{lT2}.





\section{Analysis of the process $(M_{\theta_\de})_{\de\geq 0}$}
\label{sec41}

\subsection{Basic properties}
\label{sec410}
 Let the underlying diffusions $X$ be such that $\theta_\de$ is  finite a.s. for all values on $\theta$. Then, by the contnuity of the paths of $X$, both $\theta_\de$ and $ M_{\theta_\de}$ tend to $\infty$ as $\de\to\infty$.  Recall that $r=+\infty$, and to fix ideas we take $l=-\infty$.    It is assumed also that $X(0)= 0$ which implies that $M_{\theta_0}=0$.  Clearly, $M_\theta$ is a non-decreasing, right continuous, pure jump process. On every compact subinterval of $(0,\infty)$ $M_\theta$  undertakes finitely many jumps and is constant between the jumps. On any interval $[0,\ep)$ with $\ep>0$ there are countably infinitely many jumps accumulating at 0 and $M_{\th_\de}\to 0$ as $\de\to 0$.

We study first the Markov property of  $(M_{\theta_\de})_{\de\geq 0}$.

\begin{proposition}
\label{j1}
For a Borel-measurable  $f:\R_+\times\R_+\mapsto\R_+$ and   $\de>\rho\geq 0$ 
\begin{equation}
\label{je01}
\E_0\left(f(M_{\theta_\de},{\theta_\de})\,|\, \cF_{\th_\rho}\right)=
Q_{\rho,\de}\lp M_{\theta_\rho},{\theta_\rho}; f\rp,
\end{equation}
where
\begin{align*}
 Q_{\rho,\de}\lp y,s; f\rp&=\E_{y-\rho}\lp f(y,s+H_{y-\de})\,;\, H_{y-\de}<H_y\rp
\\
&\hskip2cm+  \E_{y-\rho}\lp F_\de(y,H_y)\,;\, H_y< H_{y-\de}\rp
\end{align*}
and for $a>0$
\begin{align*}
 F_\de(y,{a}):=\E_{y}\lp f(M_{\theta_\de},a+\th_\de)\rp.
\end{align*}
In particular, for a Borel-measurable $f:\R_\mapsto\R_+$
\begin{equation}
\label{je1}
\E_0\left(f(M_{\theta_\de})\,|\, \cF_{\th_\rho}\right)=Q_{\rho,\de}\lp M_{\theta_\rho}; f\rp,
\end{equation}
where
\begin{align}\hskip-.2cm
\label{je5}
 Q_{\rho,\de}\lp y; f\rp&:=f(y)\P_{y-\rho}\lp H_{y-\de}<H_y\rp + \E_y\lp f(M_{\theta_\de})\rp \P_{y-\rho}\lp H_y< H_{y-\de}\rp,
\end{align}
and, hence,   $(M_{\theta_\de})_{\de\geq 0}$    is a Markov process with respect to its own filtration $\cF^M_{\th}=(\cF^M_{\th_\tau})_{\tau\geq 0}:=(\sigma\{ M_{\th_u}; u\leq \tau\})_{\tau\geq 0}$.

\end{proposition}

\begin{proof}
For notational simplicity, we prove (\ref{je1}) and leave  (\ref{je01}) to the reader.
For $\de>\rho\geq 0$ introduce
\begin{align*}
&
U:=\inf\{t\geq 0\,;\, X_{t+\th_\rho}\leq M_{\th_\rho}-\de\}\ {\text{and}}\
V:=\inf\{t\geq 0\,;\, X_{t+\th_\rho}= M_{\th_\rho}\}.
\end{align*}
In case $U<V$, we have $M_{\th_\de}=M_{\th_\rho}$ and
 \begin{align*}
&\E_{0}\lp f(M_{\th_\de})\,{\bf 1}_{\{U<V\}}\,|\, \cF_{\th_\rho}\rp
=f(M_{\th_\rho})g\lp X_{\th_\rho},M_{\th_\rho}\rp
=f(M_{\th_\rho})g\lp M_{\th_\rho}-\rho,M_{\th_\rho}\rp,
 \end{align*}
where
 \begin{align}
\label{G1}
&
g(x,y)=\P_{x}\lp  H_{y-\de}<H_{y}\rp=\frac{S(y)-S(x)}{S(y)-S(y-\delta) }
 \end{align}
If $U>V$ then $\th_\de=\th_\rho+\inf\{t\geq 0\,;\, M_{t+\th_\rho}-X_{t+\th_\rho}>\de\}$  and evoking the strong Markov property at $\th_\rho$ yields
 \begin{align*}
&\E_{0}\lp f(M_{\th_\de})\,{\bf 1}_{\{U>V\}}\,|\, \cF_{\th_\rho}\rp
=\E_{X_{\th_\rho}}\lp f(M_{\th_\de})\rp\lp 1-g\lp X_{\th_\rho},M_{\th_\rho}\rp\rp.
 \end{align*}
This completes the proof of (\ref{je1}).  The Markov property of   $(M_{\theta_\de})_{\de\geq 0}$
follows from (\ref{je1}) by applying the tower property of conditional expectations:
\begin{align*}
\E_0\left(f(M_{\theta_\de})\,|\, \cF^M_{\th_\rho}\right)&=
\E_0\left(\E_0\left(f(M_{\theta_\de})\,|\, \cF_{\th_\rho}\right)\,|\, \cF^M_{\th_\rho}\right)
=Q_{\rho,\de}\lp M_{\theta_\rho};f\rp,
\end{align*}
and, similarly,
\begin{align*}
\E_0\left(f(M_{\theta_\de})\,|\, M_{\th_\rho}\right)&=Q_{\rho,\de}\lp M_{\theta_\rho};f\rp.
\end{align*}
\end{proof}

\begin{corollary}
\label{jo1}
For $\de>\rho\geq 0$,  and if   $M_{\theta_\rho}> v$ then
\begin{align}
\label{e5200}
&
\P_0\left(M_{\theta_\de}>v\,|\, M_{\th_\rho}\right)=1,
\end{align}
and if $M_{\theta_\rho}\leq v$ then for $y\leq v$
\begin{align}
\label{e520}
&
\nonumber
\P_0\left(M_{\theta_\de}>v\,|\, M_{\th_\rho}=y\right)\\
&\hskip2cm
=\frac{S(y-\rho)-S(y-\de)}{S(y)-S(y-\de)}\exp\l(-\int_y^v\frac{dS(z)}{S(z)-S(z-\delta)}\r).
\end{align}

\end{corollary}

\begin{proof}
The claim in (\ref{e5200}) when $M_{\theta_\rho}>v$ is obvious since $\rho\mapsto M_{\theta_\rho}$ is nondecreasing. For (\ref{e520}) notice that the function $Q_{\rho,\de}$ in (\ref{je5}) can be rewritten as
\begin{align}
\label{Q1}
\nonumber
 Q_{\rho,\de}\lp y; f\rp&=f(y) + \big( \E_y\lp f(M_{\theta_\de})\rp-f(y)\big) \P_{y-\rho}\lp H_y< H_{y-\de}\rp\\
 &=f(y) + \lp \E_y\big( f(M_{\theta_\de})\rp-f(y)\big) \frac{S(y-\rho)-S(y-\de)}{S(y)-S(y-\de)}.
\end{align}
Substituting  $f(\cdot)={\bf 1}_{(v,+\infty)}(\cdot)$ in (\ref{je1}) and using (\ref{Q1})  yield in case $M_{\theta_\rho}\leq v$
\begin{equation*}
 Q_{\rho,\de}\lp y; f\rp=\P_y\lp M_{\theta_\de}>v\rp \frac{S(y-\rho)-S(y-\de)}{S(y)-S(y-\de)},
\end{equation*}
where $y=M_{\th_\rho}$. Formula (\ref{e520}) follows when applying (\ref{eq:ei01}).
\end{proof}

\begin{corollary}
\label{jo11}
For $\de>\rho> 0$  and $v>y>0$
\begin{align}
\label{e5201}
&
\nonumber
\P_0\left(M_{\th_\rho}>y, M_{\theta_\de}>v\right)\\
&\hskip2cm
=\exp\l(-\int_0^y\frac{dS(z)}{S(z)-S(z-\rho)}-\int_y^v\frac{dS(z)}{S(z)-S(z-\delta)}\r).
\end{align}
\end{corollary}
\begin{proof} The claim follows by straightforward  calculations using the conditional probability given in (\ref{e520}) and the following  expression for the $\P_0$-density of $M_{\theta_\rho}$
\begin{align}
\label{density00}
&f_{M_{\th_\rho}}(y)=\frac{S'(y)}{S(y)-S(y-\rho)}\exp\l(-\int_0^y\frac{S(dt)}{S(t)-S(t-\rho)}\r)
\end{align}
obtained from (\ref{eq:ei01}).
\end{proof}

Next we calculate the generator of  $(M_{\theta_\rho})_{\rho\geq 0}$.
It is assumed for the rest of this section that the scale function $S$ is in $C^1$.

\begin{proposition}
\label{ge1}
 For $\rho>0$ and a measurable function  $f$ with compact support  it holds a.s.
\begin{equation}
\label{e540}
\lim_{\de\downarrow\rho}\frac{\E_0(f(M_{\th_\de})\,|\, \cF_{\th_\rho})-f(M_{\th_\rho})}{\de-\rho}
=\cA_\rho f(M_{\th_\rho}),
\end{equation}
where $\cA_\rho$ is  given by
\begin{align*}
&\cA_\rho f(y):=\frac{S'(y-\rho)}{S(y)-S(y-\rho)}\int_{y}^{r}\frac{f(z)-f(y)}{S(z)-S(z-\rho)}\\
&\hskip5cm \times\exp\left(-\int_{y}^{z}\frac{S(dt)}{S(t)-S(t-\rho)}\right)
 dS(z).
\end{align*}
\end{proposition}

\begin{proof} Using (\ref{je1}) and (\ref{Q1}) the claim follows from fairly straightforward calculations. We skip the details.
\end{proof}
\noindent

\begin{remark}
\label{rmk00}
From Proposition \ref{ge1} we concude that the jump measure  of    $(M_{\theta_\rho})_{\rho\geq 0}$
is given for $z>0$ by
\begin{align*}
&\nu_{y,\rho}(dz)=\frac{S'(y-\rho)}{S(y)-S(y-\rho)}\frac{S'(y+z)}{S(y+z)-S(y+z-\rho)}\\
&\hskip5cm \times\exp\left(-\int_{y}^{y+z}\frac{S(dt)}{S(t)-S(t-\rho)}\right)\, dt.
\end{align*}
\end{remark}

For  $\rho>0$ let
\begin{align}
\label{T1}
&T^+_\rho:=\inf\{s>\rho\,:\, M_{\th_s}>M_{\th_\rho}\},
\end{align}
i.e., $T^+_\rho$ is the first jump ``time''  for the process $M_{\theta}$ after ``time''  $\rho.$ We introduce also the "companion" of $T^+_\rho$ via
\begin{align}
\label{L1}
&T^-_\rho:=\sup\{s<\rho\,:\, M_{\theta_s}<M_{\theta_\rho}\} =\inf\{s<\rho\,:\, M_{\theta_s}=M_{\theta_\rho}\}.
\end{align}
From the Markov property of $(M_{\theta_\rho})_{\rho\geq 0}$ it follows that   $T^+_\rho$ and  $T^-_\rho$ are conditionally independent given $M_{\theta_\rho}$.  The objective now is to calculate the conditional laws of  $T^+_\rho$ and  $T^-_\rho$. We have the following result.

\begin{proposition}
\label{LAST}
For $\de>\rho$
\begin{align}
\label{T2}
&\P_0\lp T^+_\rho>\de\,|\, M_{\theta_\rho}=y\rp=
 \frac{S(y)-S(y-\rho)}{S(y)-S(y-\de)},
\end{align}
and for $0<\de<\rho$
\begin{align}
\label{L2}
\nonumber
&\P_0\lp T^-_\rho<\de\,|\, M_{\theta_\rho}=y\rp\\
&\hskip2cm =\exp\left(-\int_{0}^{y}\frac{S(dt)}{S(t)-S(t-\de)} + \int_{0}^{y}\frac{S(dt)}{S(t)-S(t-\rho)}\right).
\end{align}
\end{proposition}
\begin{proof}
Since
\begin{align}
\label{TT2}
&\P_0\lp T^+_\rho>\de\,|\, M_{\theta_\rho}=y\rp=
\P_0\lp M_{\theta_\de}=y\,|\, M_{\theta_\rho}=y\rp
\end{align}
the claim concerning $T^+_\rho$ follows from (\ref{e520}) by taking therein $v=y$.
For $T^-_\rho$ we have with  $0<\de<\rho$
\begin{align*}
\P_0\lp T^-_\rho<\de\,|\, M_{\theta_\rho}=y\rp &=\P_0\lp M_{\th_\de}=M_{\th_\rho}\,|\, M_{\th_\rho}=y\rp\\
& =\P_0\lp M_{\th_\de}=M_{\th_\rho},\, M_{\th_\rho}\in dy\rp / \P_0\lp M_{\th_\rho}\in dy\rp\\
& =\P_0\lp M_{\th_\de}=M_{\th_\rho}\, |\, M_{\th_\de}=y\rp f_{M_{\th_\de}}(y) / f_{M_{\th_\rho}}(y),
\end{align*}
where the densities  $f_{M_{\th_\de}}$ and  $f_{M_{\th_\de}}$ are as given in  (\ref{density00})).  Applying now (\ref{TT2}) yields  the claimed formula.
\end{proof}

For $\rho>0$ and $T^+_\rho$ as in (\ref{T1}) let  $J_\rho$ denote the size of the jump at time $T^+_\rho$, i.e.,
\begin{align}
\label{jsize}
&J_\rho:=M_{\th_{T^+_\rho}}-M_{\th_{\rho}}.
\end{align}
The joint (conditional) distribution of  $T^+_\rho$ and $J_\rho$ is given in the next proposition.

\begin{proposition}
\label{JOINT}
For $\de>\rho$ and $z>0$
\begin{align}
\label{J2}
&
\nonumber
\P_0\lp T^+_\rho<\de, J_\rho>z\,|\, M_{\theta_\rho}=y\rp
=\lp S(y)-S(y-\rho)\rp\\
&\hskip1cm\times
\int_\rho^\de \frac{S'(y-u)}{\lp S(y)-S(y-u)\rp^2}\exp\l(-\int_y^{y+z}\frac{S(dt)}{S(t)-S(t-u)}\r)du.
\end{align}
\end{proposition}
\begin{proof} From (\ref{T2}) the conditional density of $T^+_\rho$ given that $M_{\theta_\rho}=y$ is calculated to be 
\begin{align}
\label{DT}
&
f_{T^+_\rho}(u)= {\bf 1}_{\{u>\rho\}} \frac{S'(y-u)\lp S(y)-S(y-\rho)\rp}{\lp S(y)-S(y-u)\rp^2}.
\end{align}
Next let
$
\th':= \inf\{t\geq \th_\rho\,:\, X_t=M_{\theta_\rho}\},
$
and consider
\begin{align*}
&
\P_0\lp  J_\rho>z\,|\, \cF_{\th'}\rp=
\P_0\lp M_{\th_{T^+_\rho}}>z+M_{\th_{\rho}} \,|\, \cF_{\th'}\rp
=:F\lp M_{\th_{\rho}},T^+_\rho\rp,
\end{align*}
where (cf. (\ref{eq:ei01}))
\begin{align}
\label{DT1}
&
F\lp y,u\rp=\P_y\lp M_{\th_u}>z+y\rp=\exp\l(-\int_y^{y+z}\frac{S(dt)}{S(t)-S(t-u)}\r).
\end{align}
Observing that
\begin{align*}
\P_0\lp T^+_\rho<\de, J_\rho>z\,|\, M_{\theta_\rho}=y\rp
&=\P_0\lp \P_0\lp T^+_\rho<\de, J_\rho>z\,|\, \cF_{\th'}\rp\,|\, M_{\theta_\rho}=y\rp\\
&=\E_0\lp {\bf 1}_{\{T^+_\rho<\de\}}F\lp M_{\th_{\rho}},T^+_\rho\rp\,|\, M_{\theta_\rho}=y\rp,
\end{align*}
and using herein (\ref{DT}) and (\ref{DT1}) yields  the claim.
\end{proof}

\begin{remark}
 \label{appendix}
Results presented above for the process $(M_{\theta_\theta})_{\theta\geq 0}$ can be seen in the general  framework of piecewise constant real-valued strong Markov processes. Indeed, let  $Y=(Y_t)_{t\geq 0}$ be such a process which at time $s\geq 0$ is located at $x\in\R$. Define
\begin{equation}
\label{Dr}
{\cal {T}}_{x,s}:=\inf\{ t\geq 0\,:\, Y_{s+t}\not= x\},
\end{equation}
and notice that the first time the process leaves the state $x$ after time $s$ is then $s+{\cal {T}}_{x,s}$. Define also for $t\geq 0$
\begin{align*}
\phi_s(t):=\phi_{x,s}(t)&:=\P({\cal {T}}_{x,s}\geq t)
\end{align*}
Then using  the strong Markov property it is seen that
for  all $t,u\geq 0$,
\begin{align}
\label{lemma1}
\phi_s(t+u)&=\phi_s(t)\phi_{t+s}(u)=\phi_s(u)\phi_{u+s}(t).
\end{align}
Consequently,  if for all $s\geq 0$
\begin{align*}
&\phi'_s(0+):=\lim_{\varepsilon\downarrow 0}\frac{\phi_s(\varepsilon)-1}{\varepsilon}
\end{align*}
 exists and is locally integrable then
 \begin{align}
 \label{lemma3}
 \phi_s(t)&=\exp\lp\int_0^t\phi'_{r+s}(0+)dr\rp.
\end{align}
To see that the formulas (\ref{lemma1}) and  (\ref{lemma3}) are indeed valid in our particular  case let $Y_t:=M_{\theta_t}$. Introduce  
\[
{\cal {T}}_{y,\rho}:=\inf\{ t\geq 0 : M_{\theta_{t+\rho}}>y\}
\]
and
\[
\phi_\rho (t):=\P({\cal {T}}_{y,\rho}\geq t\,|\, Y_\rho=y),
\]
where $y= M_{\theta_\rho}$. Then with  $T^+_\rho$ as defined in (\ref{T1}) we have
\[
{\cal {T}}_{y,\rho}=T^+_\rho-\rho,
\]
and
\[
\phi_\rho (t)=\P_0(T^+_\rho\geq \rho+t|M_{\theta_\rho}=y).
\] 
From (\ref{T2}) 
\begin{equation}
\label{lemma1A}
\phi_\rho (t)=\frac{S(y)-S(y-\rho)}{S(y)-S(y-t-\rho)},
\end{equation}
and it is straightforward to check that (\ref{lemma1}) holds for the expression on the right hand side of (\ref{lemma1A}). If $S'$ exists we have 
\begin{align*}
&\phi'_{\rho+s}(0+)=-\frac {S'(y-\rho-s)}{S(y)-S(y-\rho+s)},
\end{align*}
and it is easily seen that 
\begin{align*}
\exp\lp-\int_0^t\frac {S'(y-\rho-r)}{S(y)-S(y-r-\rho)}dr\rp= \frac{S(y)-S(y-\rho)}{S(y)-S(y-t-\rho)}
\end{align*}
i.e., the expression on the right hand side of (\ref{lemma1A}) satisfies  (\ref{lemma3}). Similar analysis can be performed concerning (\ref{L2}) and (\ref{J2}), but we do not go into these details.

\end{remark}

\subsection{Comparision of  $(M_{\theta_\de})_{\de\geq 0}$ and  $(D^-_{H_\alpha})_{\al\geq 0}$  }
\label{sec42}

From the strong Markov property of $X$ it follows easily, as observed in (\cite{salminenvallois20}),  that   $(D^-_{H_\al})_{\al\geq 0}$ is also Markov with respect to its own filtration. Moreover, it is seen from the proof of  Proposition 4.1 in ibid.
 that for $\alpha\geq\beta>0$ and a measurable and bounded function $f$
\begin{equation}
\label{ke1}
\E_0\left(f(D^-_{H_\alpha})\,|\, \cF_{H_\beta}\right)=\E_0\left(f(D^-_{H_\alpha})\,|\, D^-_{H_\beta}\right)=\widehat Q_{\beta,\alpha}\lp D^-_{H_\beta}; f\rp,
\end{equation}
where
\begin{align}\hskip-.2cm
\label{ke5}
\widehat Q_{\beta,\alpha}\lp y; f\rp&:=f(y)\P_{\beta}\lp D^-_{H_\alpha}\leq y\rp + \E_\beta\lp f(D^-_{H_\alpha})\,;\, D^-_{H_\alpha}>y\rp.
\end{align}
Notice the structural  resemblance of the formulas (\ref{ke5}) and (\ref{je5}). From  the distribution function of $D^-_{H_\alpha}$ as  given in (\ref{eqSaVa}) we can calculate the density of $D^-_{H_\alpha}$, and then find an explicit expression of the semigroup $\widehat Q$ in terms of the scale function of $X$. 

We present here an alternative approach for this expression based solely on the distributions associated  with $M_{\theta_\delta},\, \delta>0$.  Recalling  (\ref{DM1}),  the density of $D^-_{H_\alpha}$ can be computed from the distribution function of $M_{\theta_\delta}$ given in  (\ref{eq:ei01}).  Indeed, taking the derivative therein with respect $\delta$, assuming $l=-\infty$  and  $X(0)=0$ we have
\begin{align}
\label{e520x}
&
\nonumber
\hskip-1cm\P_0\lp D^-_{H_\alpha}\in d\delta\rp/d\delta
=\int_0^\alpha\frac{S'(z-\delta)S'(z)}{(S(z)-S(z-\delta))^2}\,dz
\\&\hskip4cm\times
\exp\l(-\int_0^\alpha\frac{dS(z)}{S(z)-S(z-\delta)}\r).
\end{align}
To proceed,  we extend  (\ref{DM1}) to get
\begin{equation}\label{DM2}
   \{ D^-_{H_\alpha}< \delta, D^-_{H_\beta}< \rho,\}=\{M_{\theta_\delta}>\alpha, M_{\theta_\rho}> \beta\}.
\end{equation}
Then Corollary \ref{jo11} yields the following result.
\begin{proposition}
\label{jo111}
For $\de\geq\rho> 0$  and $\alpha>\beta>0$
\begin{align}
\label{e5202}
&
\nonumber
\P_0\left(D^-_{H_\alpha}< \delta, D^-_{H_\beta}< \rho\right)\\
&\hskip2cm
=\exp\l(-\int_0^\beta\frac{dS(z)}{S(z)-S(z-\rho)}-\int_\beta^\alpha\frac{dS(z)}{S(z)-S(z-\delta)}\r).
\end{align}
\end{proposition}
\noindent
Differentiating in (\ref{e5202}) with respect to $\rho$ yields
\begin{align}
\label{e5203}
&
\nonumber
\P_0\left(D^-_{H_\alpha}< \delta, D^-_{H_\beta}\in d\rho\right)/d\rho
=\int_0^\beta\frac{S'(z-\rho)S'(z)}{(S(z)-S(z-\rho))^2}\,dz\\
&\hskip2cm
\times\exp\l(-\int_0^\beta\frac{dS(z)}{S(z)-S(z-\rho)}-\int_\beta^\alpha\frac{dS(z)}{S(z)-S(z-\delta)}\r).
\end{align}
From  (\ref{e520x}) we deduce  an explicit form of the density of $D^-_{H_\beta}$ and using this in  (\ref{e5203}) yields   the conditional law describing the semigroup of  $(D^-_{H_\eta})_{\eta\geq 0}$. This is   stated in the next proposition.
\begin{proposition}
\label{jo112}
For  $\alpha>\beta>0$
\begin{align}
\label{X}
&
\nonumber
\P_0\left(D^-_{H_\alpha}\leq \delta\,|\,  D^-_{H_\beta}= \rho\right)
=\begin{cases}
\displaystyle{\exp\l(-\int_\beta^\alpha\frac{dS(z)}{S(z)-S(z-\delta)}\r)},& \de\geq\rho,\\
0,& \de<\rho\\
\end{cases}
\end{align}
\end{proposition}

Finally, as indicated above the semigroup of  $(D^-_{H_\al})_{\al\geq 0}$ is given in (4.1) in \cite{salminenvallois20},  it is possible, therefore, to work also vice versa, that is, to deduce the semigroup of  $(M_{\theta_\de})_{\de\geq 0}$ from the semigroup of  $(D^-_{H_\al})_{\al\geq 0}$. Notice, moreover, that the Markov property  of $(M_{\theta_\de})_{\de\geq 0}$ follows from the strong Markov property of $X$, as shown in the proof of Proposition \ref{j1}.

\vskip1cm
\noindent
{\bf Acknowledgement}. We wish to thank Patrick Fitzsimmons for pointing out the resemblance between the distribution of $M_\theta$ derived by Lehoczky in \cite{lehoczky77} and the distribution of $D_{H_\eta}$ derived in our paper \cite{salminenvallois20}. This observation triggered the research reported in  the present paper.

\bibliographystyle{plain}
\bibliography{yor1}
\end{document}